\documentclass[reqno]{amsart}
\usepackage{amssymb,eucal,latexsym, mathrsfs,enumerate,graphicx,tikz-cd}

\newenvironment{enumeratei}{\begin{enumerate}[\upshape (i)]}%
{\end{enumerate}}
{\end{enumerate}}
\newenvironment{enumerater}{\begin{enumerate}[\upshape (1)]}%
{\end{enumerate}}

\newcommand{\pup}[1]{\textup{(}{#1}\textup{)}}
\newcommand{\lrep}{$\ell$-rep\-re\-sentable}
\newcommand{\lrepy}{$\ell$-rep\-re\-sentabil\-ity}

\newcommand{\jh}{join-ho\-mo\-mor\-phism}

\newcommand{\eqdef}{\overset{\mathrm{def}}{=}}

\newcommand{\lgrp}{$\ell$-group}
\newcommand{\lhom}{$\ell$-ho\-mo\-mor\-phism}
\newcommand{\lidl}{$\ell$-ideal}

\newcommand{\kk}{\Bbbk}
\newcommand{\kkp}[1]{\kk^{(#1)}}

\DeclareMathOperator{\Spec}{Spec}

\DeclareMathOperator{\Specr}{Spec_{r}}

\newcommand{\FL}{\operatorname{F}_{\ell}}

\DeclareMathOperator{\conv}{conv}
\DeclareMathOperator{\cone}{cone}

\DeclareMathOperator{\card}{card}

\newcommand{\ga}{\alpha}
\newcommand{\gb}{\beta}

\newcommand{\gd}{\delta}
\newcommand{\gf}{\varphi}
\newcommand{\gi}{\iota}

\newcommand{\gl}{\lambda}
\newcommand{\gq}{\theta}
\newcommand{\gs}{\sigma}
\newcommand{\go}{\omega}
\newcommand{\eps}{\varepsilon}

\newcommand{\bck}[1]{[\![{#1}]\!]}
\newcommand{\ps}[1]{[{#1}]}

\newcommand{\gS}{\Sigma}

\newcommand{\sd}{\mathbin{\smallsetminus}}

\newcommand{\jz}{$(\vee,0)$}

\newcommand{\jzh}{\jz-ho\-mo\-mor\-phism}

\newcommand{\pI}[1]{\bigl({#1}\bigr)}

\newcommand{\set}[1]{\left\{#1\right\}}
\newcommand{\setm}[2]{\set{{#1}\mid{#2}}}
\newcommand{\vecm}[2]{\left({#1}\mid{#2}\right)}
\newcommand{\seq}[1]{\langle{#1}\rangle}

\newcommand{\es}{\varnothing}
\newcommand{\res}{\mathbin{\restriction}}

\newcommand{\QQ}{\mathbb{Q}}

\DeclareMathOperator{\Op}{Op}
\newcommand{\Ops}{\Op^-}

\DeclareMathOperator{\Id}{Id}
\DeclareMathOperator{\Idc}{Id_c}

\newcommand{\cA}{{\mathcal{A}}}

\newcommand{\cD}{{\mathcal{D}}}

\newcommand{\cK}{{\mathcal{K}}}

\newcommand{\cN}{{\mathcal{N}}}
\newcommand{\cO}{{\mathcal{O}}}

\newcommand{\cR}{{\mathcal{R}}}

\numberwithin{equation}{section}

\newtheorem*{stat}{\name}
\newcommand{\name}{testing}

\newenvironment{all}[1]{\renewcommand{\name}{#1}\begin{stat}}
                        {\end{stat}}

\theoremstyle{plain}

\newtheorem{theorem}{Theorem}[section]

\newtheorem{corollary}[theorem]{Corollary}
\newtheorem{lemma}[theorem]{Lemma}

\newtheorem{case}{Case}
\newtheorem*{sclaim}{Claim}

\theoremstyle{definition}

\newtheorem{definition}[theorem]{Definition}
\newtheorem{notation}[theorem]{Notation}

\newtheorem{problem}[theorem]{Problem}

\theoremstyle{remark}
\newtheorem{remark}[theorem]{Remark}

\newcommand{\qedc}{{\qed}~{\rm Claim~{\theclaim}.}}
\newcommand{\qedsc}{{\qed}~{\rm Claim.}}

\newenvironment{scproof}
{\begin{proof}[Proof of Claim.]}
{\qedsc\renewcommand{\qed}{}\end{proof}}

\numberwithin{figure}{section}
\numberwithin{table}{section}

\newcommand{\bc}{\boldsymbol{c}}
\newcommand{\bd}{\boldsymbol{d}}
\newcommand{\be}{\boldsymbol{e}}

\newcommand{\scL}{\mathbin{\mathscr{L}}}

\newcommand{\dlat}{distributive lattice}
\newcommand{\dzlat}{distributive $0$-lattice}
\newcommand{\cn}{completely normal}

\newcommand{\scp}[2]{({#1}\mid{#2})}

\title[The MV-spectrum Problem in cardinality aleph one]{A solution to the MV-spectrum Problem in size aleph one}

\author[M. Plo\v{s}\v{c}ica]{Miroslav Plo\v{s}\v{c}ica}
\address{Faculty of Natural Sciences\\
\v{S}af\'arik's University\\
Jesenn\'a 5\\
04154 Ko\v{s}ice\\
Slovakia}
\email{miroslav.ploscica@upjs.sk}
\urladdr{https://ploscica.science.upjs.sk}

\author[F. Wehrung]{Friedrich Wehrung}
\address{Normandie Universit\'e, UNICAEN\\
CNRS UMR 6139, LMNO\\
14000 Caen\\
France}
\email{friedrich.wehrung01@unicaen.fr}
\urladdr{https://wehrungf.users.lmno.cnrs.fr}
\date{\today}



\begin{document}

\begin{abstract}
Denote by~$\Idc{G}$ the lattice of all principal \lidl{s} of an Abelian \lgrp~$G$.
Our main result is the following.
\begin{all}{Theorem}
For every countable Abelian \lgrp~$G$, every countable \cn\ \dzlat~$L$, and every closed $0$-lattice homomorphism $\gf\colon\Idc{G}\to L$, there are a countable Abelian \lgrp~$H$, an \lhom\ $f\colon G\to H$, and a lattice isomorphism $\gi\colon\Idc{H}\to L$ such that $\gf=\gi\circ\Idc{f}$.
\end{all}
We record the following consequences of that result:
\begin{enumerater}
\item\label{ArrLift}
A $0$-lattice homomorphism $\gf\colon K\to L$, between countable \cn\ \dzlat{s}, can be represented, with respect to the functor~$\Idc$, by an \lhom\ of Abelian \lgrp{s} if{f} it is closed.

\item\label{CBDChar}
A \dzlat~$D$ of cardinality at most~$\aleph_1$ is isomorphic to some~$\Idc{G}$ if{f}~$D$ is \cn\ and for all $a,b\in D$ the set $\setm{x\in D}{a\leq b\vee x}$ has a countable coinitial subset.
This solves Mun\-di\-ci's MV-spectrum Problem for cardinalities up to~$\aleph_1$.
The bound~$\aleph_1$ is sharp.
\end{enumerater}
Item~\eqref{ArrLift} is extended to commutative diagrams indexed by forests in which every node has countable height.
All our results are stated in terms of vector lattices over any countable totally ordered division ring.
\end{abstract}

\maketitle


\section{Introduction}\label{S:Intro}
The set $\Spec{G}$ of all prime \lidl{s} in an Abelian lattice-ordered group (in short \emph{\lgrp})~$G$ with order-unit, endowed with the hull-kernel topology, is a spectral space (as defined in Hochster~\cite{Hoch1969}) called the \emph{spectrum of~$G$}.
The topological spaces~$\Spec{G}$ satisfy an additional property called \emph{complete normality}, which states that the specialization order is a root system.
The problem of characterizing all spaces~$\Spec{G}$ is stated in the following references:
\begin{itemize}
\item
In his 1973 paper, Mart{\'\i}nez \cite[Question~II]{Mart1973} asks (in an equivalent form) for a characterization of isomorphic copies of \lidl\ lattices of \emph{Archimedean} \lgrp{s}.

\item
In his 2011 monograph, Mundici \cite[Problem~2]{Mund2011} asks ``Which topological spaces are homeomorphic to $\Spec{A}$ for some MV-algebra~$A$''.
Since an MV-algebra and its associated Abelian \lgrp\ with unit (through Mundici's equivalence~\cite{Mund86}) have homeomorphic spectra, this is equivalent to the corresponding question about \lgrp{s} with order-unit.

\end{itemize}

Beginning in Mart{\'\i}nez' paper~\cite{Mart1973}, most known answers to those problems are stated in terms of the Stone dual of~$\Spec{G}$, which is the lattice~$\Idc{G}$ of all principal \lidl{s} of~$G$ (cf. Keimel~\cite{Keim1971}).
Such lattices will be called \emph{\lrep}.
They are both bounded and distributive.
Owing to Monteiro~\cite{Mont1954}, complete normality of the topological space~$\Spec{G}$ translates to a first-order lattice-theoretical property of~$\Idc{G}$, also called complete normality, namely
 \begin{equation*}
 (\forall a,b)(\exists x,y)(a\vee b=a\vee y=x\vee b
 \text{ and }x\wedge y=0)\,.
 \end{equation*}
Delzell and Madden's non-\lrep\ \cn\ lattice from~\cite{DelMad1994} has~$\aleph_1$ elements.
On the other hand, a \cn\ bounded distributive lattice~$D$ is \lrep\ provided either~$D$ is a dual Heyting algebra (cf.
Cignoli \emph{et al.}~\cite{CGL}, Iberkleid \emph{et al.}~\cite{IMM2011}) or~$D$ is countable%
\footnote{
Throughout the paper ``countable'' will mean ``at most countable''.
}
(cf. Wehrung~\cite{MV1}).

There are also negative results, whose spirit is ``\emph{\lrepy\  cannot be characterized}'' [in a logically simpler way than originally defined].
Lenzi and Di Nola observe in~\cite{LenDiN2020} that since the class of all \lrep\ lattices is not closed under ultraproducts, it is not the class of models of any existential second-order sentence.
In Wehrung~\cite{NonElt} it is proved, building on Wehrung~\cite{Ceva} and an extension of the condensate construction initiated in Gillibert and Wehrung~\cite{Larder}, that the class of all \lrep\ lattices is not the class of all models of any class of sentences in the infinitary language~$\scL_{\infty\gl}$ for fixed%
\footnote{
One cannot do better than fixing~$\gl$, because any class of models closed under isomorphic copy is the class of all models of its~$\scL_{\infty\infty}$ consequences.
}
~$\gl$.
The subsequent paper Wehrung~\cite{AccProj} establishes that the class of all \lrep\ lattices is not the class of models of any sentence of the form $(\forall X)\gf$ where~$X$ is a second-order variable with bounded arity and~$\gf$ is an~$\scL_{\infty\infty}$ formula (we say \emph{co-projective}\label{co-proj}).
Those non-rep\-re\-sentabil\-ity results extend to \emph{Archimedean} \lgrp{s}; however, it is still unknown whether they extend to Archimedean \lgrp{s} \emph{with order-unit} (cf. Problem~\ref{Pb:ArchUnit}).

Those negative results also apply to the larger class of all \emph{homomorphic images of \lrep\ lattices}, and in fact to any intermediary class.

A strong necessary condition for \lrepy\ was coined, under different names (``$\Id\go$'' and ``$\gs$-Conrad'', respectively), in Cignoli \emph{et al.}~\cite{CGL} and  Iberkleid \emph{et al.}~\cite{IMM2011}; it got renamed again in Wehrung~\cite{MV1} as follows.

\begin{definition}\label{D:CBD}
A lattice~$L$ has \emph{countably based differences} if for all $a,b\in L$ there exists a countable subset~$C$ of~$L$ such that for every $x\in L$, $a\leq b\vee x$ if{f} $c\leq x$ for some $c\in C$.
\end{definition}

Every \lrep\ lattice has countably based differences.
Delzell and Madden's counterexample from~\cite{DelMad1994} does not have countably based differences.
The main counterexample~$D$ in Wehrung~\cite{Ceva} has size~$\aleph_2$;
it is \cn, has countably based differences, but is not what is called there \emph{Cevian}.
Being Cevian is an existential second-order property.
Homomorphic images of \lrep\ lattices are always Cevian, thus~$D$ is not such a lattice.
Plo\v{s}\v{c}ica's non \lrep\ counterexample from~\cite{Plo21} is \cn, has countably based differences, and is Cevian; it has cardinality $(2^{\aleph_0})^+$.
Those results show that in some sense, \emph{characterizing \lrepy\ beyond the size~$\aleph_2$ is hopeless}.

In this note we deal with the remaining gap, that is, the cardinality~$\aleph_1$.
The main consequence, pertaining to the MV-spectrum Problem, of our work is the following extension, from countable to~$\aleph_1$, of the main result of Wehrung~\cite{MV1}, contained in Corollary~\ref{C:Main1}:

\begin{all}{Theorem}
A \dzlat\ of cardinality at most~$\aleph_1$ is \lrep\ if{f} it is \cn\ and has countably based differences.
\end{all}

\emph{Via} Stone duality, this result is a full solution of the MV-spectrum Problem for spaces with at most~$\aleph_1$ compact open subsets.
Due to the main counterexample from Wehrung~\cite{Ceva}, the bound~$\aleph_1$ in this result is sharp (e.g., the theorem above does not extend to the size~$\aleph_2$).

Our main result, Theorem~\ref{T:Main}, deals with countable structures only.
It also involves the concept of a \emph{closed homomorphism} between lattices;
a lattice homomorphism $f\colon K\to L$ is \emph{closed} if for all $a,b\in K$ and all $x\in L$, if $f(a)\leq f(b)\vee x$ then there exists $u\in K$ such that $a\leq b\vee u$ whereas $f(u)\leq x$.
In the context of Stone duality between distributive $0$-lattices with cofinal $0$-lattice homomorphisms and generalized spectral spaces with spectral maps (cf. Rump and Yang~\cite{RumYan2009}), it can be proved that a cofinal $0$-lattice homomorphism is closed if{f} its Stone dual is closed in the topological sense (i.e., it sends closed sets to closed sets); see Iberkleid \emph{et al.} \cite[Lemma~1.3.5]{IMM2011}.

For any \lhom\ $f\colon G\to H$ between Abelian \lgrp{s}, the lattice homomorphism $\Idc{f}\colon\Idc{G}\to\Idc{H}$ is closed (cf. Wehrung \cite[Proposition~2.6]{MV1}).
Our main result implies that conversely, \emph{every closed homomorphism between countable \cn\ \dzlat{s} can be represented in this way}.
In fact, that result extends to commutative diagrams indexed by forests in which every node has countable height (cf. Theorem~\ref{T:Main+}).
Our main result, illustrated on Figure~\ref{Fig:MainStat}, states as follows:

\begin{figure}[htb]
\begin{tikzcd}[row sep=tiny]
\centering
\Idc{G}\arrow[dd,"\gf"',"\text{closed}"] && G\arrow[dd,"f"] &
\Idc{G}\arrow[dd,"\Idc{f}"']\arrow[ddr,"\gf"] &\\
& \scalebox{2}{\ensuremath\rightsquigarrow} &&&\\
L && H & \Idc{H}\arrow[r,"\gi","\cong"'] & L\end{tikzcd}
\caption{Illustrating Theorem~\ref{T:Main}}
\label{Fig:MainStat}
\end{figure}

\begin{all}{Theorem~\ref{T:Main} (restated for \lgrp{s})}
Let~$G$ be a countable Abelian \lgrp, let~$L$ be a countable \cn\ \dzlat, and let $\gf\colon\Idc{G}\to L$ be a closed $0$-lattice homomorphism.
Then there are a countable Abelian \lgrp~$H$, an \lhom\ $f\colon G\to H$, and a lattice isomorphism $\gi\colon\Idc{H}\to L$ such that $\gf=\gi\circ\Idc{f}$.
\end{all}

All our results will be formulated in the context of left vector lattices over \emph{countable} totally ordered division rings~$\kk$ (due to Wehrung \cite[\S~9]{RAlg}, the countability assumption on~$\kk$ cannot be dropped).
The canonical embedding from any Abelian \lgrp~$G$ into its divisible hull~$\QQ\otimes G$ is, following the terminology of Anderson and Feil~\cite{AnFe}, an Archimedean extension, so $\Idc{G}\cong\Idc(\QQ\otimes G)$.
Conversely, any vector lattice~$E$ over~$\QQ$ (or, more generally, over an Archimedean totally ordered field) has the same \lidl{s} as an \lgrp\ as a vector lattice; hence \emph{the vector lattice results entail the \lgrp\ results}.

\section{Strategy of the proof}\label{S:Strategy}
Given a countable \cn\ \dzlat~$L$, a countable left vector lattice~$E$ over a countable totally ordered division ring~$\kk$, we wish to prove that every closed $0$-lattice homomorphism $\gf\colon\Idc{E}\to L$ is, up to isomorphism (over~$\Idc{E}$), $\Idc{f}$ for some vector lattice homomorphism $f\colon E\to F$.
The initial step consists of observing that~$E$ is a homomorphic image of the free $\kk$-vector lattice~$\FL(I,\kk)$ over any countably infinite set~$I$, thus reducing the problem to the case where $E=\FL(I,\kk)$ is a free vector lattice.
Then, owing to the Baker-Bernau-Madden duality for free vector lattices (cf. Lemma~\ref{L:BakBey}), $\Idc\FL(I,\kk)$ is isomorphic to the sublattice~$\Ops\kkp{I}$ of the powerset algebra of the free left $\kk$-vector space~$\kkp{I}$ generated by all open half-spaces defined by elements of~$\kkp{I}$.

The original closed $0$-lattice homomorphism $\gf\colon\Ops\kkp{I}\to L$ is then inductively enlarged to a chain of closed homomorphisms $\gf_m\colon\Ops\kkp{I_m}\to L$, with each $I_m\eqdef I\sqcup\set{0,\dots,m-1}$ (where~$\sqcup$ denotes disjoint union).
Each map~$\gf_m$ is in turn constructed as the common extension of an ascending chain $\vecm{\gf_{m,n}}{n<\go}$ of maps $\gf_{m,n}\colon\Ops(\kkp{I_m}\cup\cD_{m,n})\to L$ for finite sets $\cD_{m,n}\subset\kkp{I_{m+1}}$.
Setting $\cD_{m,0}\eqdef\set{\gd_m}$, where~$\gd_m$ denotes the $m$th coordinate projection $\kkp{I_m}\to\kk$, the value of~$\gf_{m,0}$ at the positive open half-space defined by~$\gd_m$ is initially set at the $m$th element of~$L$ (with respect to a given enumeration).
The common extension $\psi\colon\Ops\kkp{I\sqcup\go}\to L$ of the~$\gf_m$ will then be a closed, surjective homomorphism; this easily entails the desired solution.

The construction of the~$\gf_{m,n}$ follows, to some extent, the strategy initiated for the countable case in Wehrung~\cite{MV1}.
According to the parity of~$n$, the step from~$n$ to~$n+1$ may either enlarge~$\cD_{m,n}$ by one element (``Domain Step'') or solve a ``closedness problem'' $\gf_{m,n}\ps{a}\leq\gf_{m,n}\ps{b}\vee\be$ (``Closure Step''; $\ps{a}$ denotes the positive open half-space associated to~$a$, see Notation~\ref{Not:ps1}).
The Domain Step relies on its original version, stated for finite sets of hyperplanes, stated in Wehrung \cite[Lemma~6.6]{MV1}.

However, the original approach to the Closure Step, in the form of Claim~1 from the proof of Wehrung \cite[Lemma~7.1]{MV1}, can be proved to fail in this more general context; so the argument of Plo\v{s}\v{c}ica and Wehrung~\cite{MV2} is not sufficient to go from ``homomorphic image'' to ``isomorphic copy''.
(This should be no surprise, because closed homomorphisms play no role in~\cite{MV2}.)
It turns out that the Closure Step requires a far more subtle approach.

Our first technical step, handled in Section~\ref{S:PreOOcD}, consists of stating, for a finite subset~$\cD\cup\set{c}$ of~$\kkp{J}$ where $J=I\sqcup\set{o}$, a criterion ensuring that a $0$-lattice homomorphism $\Ops(\kkp{I}\cup\cD)\to L$ can be extended to a lattice homomorphism $\Ops(\kkp{I}\cup\cD\cup\set{c})\to L$ sending the pair $(\ps{c},\ps{-c})$ to a given pair $(\bc^+,\bc^-)$ of elements in~$L$.
This criterion (cf. Lemma~\ref{L:ExtIcupD}) is stated in terms of the system of inequalities \eqref{Eq:c+-=0}--\eqref{Eq:LB2a-}, that should be satisfied by~$(\bc^+,\bc^-)$.

Our second technical step, aimed at handling a given closedness problem\newline $\gf_{m,n}\ps{a}\leq\gf_{m,n}\ps{b}\vee\be$, consists of taking advantage of $\gf_m\colon\Ops\kkp{I_m}\to L$ be closed, leading to the satisfiability of \eqref{Eq:c+-=0}--\eqref{Eq:LB2a-} with $\cD\eqdef\cD_{m,n}$ and $c\eqdef a-\gl b$ for large enough~$\gl\in\kk$, together with the additional inequality $\bc^+\leq\gf_{m,n}\ps{-b}\vee\be$.
This yields a solution to our closedness problem at the next step~$\gf_{m,n+1}$, and thus the desired conclusion (cf. Section~\ref{S:ClosStep}).

\section{Basic concepts and terminology}\label{S:Basic}

A subset~$C$ in a poset~$P$ is \emph{coinitial} if every element of~$P$ lies above some element of~$C$.
We say that~$P$ is a \emph{forest} if every principal ideal of~$P$ is well-ordered under the induced order.

For any structure~$G$ with a partially ordered Abelian group reduct (e.g., an ordered vector space, an ordered field, and so on), we set $G^+\eqdef\setm{x\in G}{0\leq x}$.
We set $x^+\eqdef x\vee 0$ and $|x|\eqdef x\vee(-x)$ if those elements exist, whenever $x\in G$.

For a totally ordered division ring~$\kk$, a \emph{$\kk$-vector lattice} is a left $\kk$-vector space~$E$ endowed with a translation-invariant lattice structure such that $\kk^+E^+\subseteq E^+$.
A map $f\colon E\to F$, between vector lattices, is a \emph{$\kk$,\lhom} if it is both a vector space homomorphism and a lattice homomorphism.
An \emph{\lidl} of~$E$ is a vector subspace of~$E$ closed under the lattice operations.
Principal (equivalently, finitely generated) \lidl{s} are the subsets of the form
$\seq{a}\eqdef\setm{x\in E}{(\exists\gl\in\kk^+)(|x|\leq\gl|a|)}$ for $a\in E$.
For any $a,b\in E^+$, $\seq{a}\vee\seq{b}=\seq{a\vee b}=\seq{a+b}$ whereas $\seq{a}\wedge\seq{b}=\seq{a\wedge b}$.
The lattice~$\Id{E}$ of all \lidl{s} of~$E$ is a distributive algebraic lattice, of which its semilattice of compact elements, $\Idc{E}\eqdef\setm{\seq{x}}{x\in E^+}$, is a $0$-sublattice.
Any $\kk$,\lhom\ $f\colon E\to F$ gives rise to a $0$-lattice homomorphism $\Idc{f}\colon\Idc{E}\to\Idc{F}$, $\seq{x}_E\mapsto\seq{f(x)}_F$.
The assignment $E\mapsto\Idc{E}$, $f\mapsto\Idc{f}$ is a functor.
This functor preserves all directed colimits.

\section{Consonance and closed homomorphisms}\label{S:ClHom}

The following two definitions originate in Wehrung~\cite{MV1} and subsequent papers, where they are widely used.
We recall them for convenience.

\begin{definition}\label{D:Closedat}
Let~$K$ and~$L$ be lattices and let $f\colon K\to L$ be a lattice homomorphism.
We say that~$f$ is \emph{closed at} a pair $(a,b)\in K\times K$ if for all $x\in L$ such that $f(a)\leq f(b)\vee x$ there exists $u\in K$ such that $a\leq b\vee u$ and $f(u)\leq x$.
\end{definition}

In particular, $f$ is closed if{f} it is closed at every pair in $K\times K$.

\begin{definition}\label{D:Cons}
A pair $(a,b)$ of elements in a \dzlat~$L$ is \emph{consonant} in~$L$ if there is a pair $(u,v)\in L\times L$ \pup{then called a \emph{splitting pair} for~$(a,b)$} such that $a\vee b=a\vee v=u\vee b$ whereas $u\wedge v=0$; or, equivalently, $a=(a\wedge b)\vee u$, $b=(a\wedge b)\vee v$, and $u\wedge v=0$.
A subset~$X$ of~$L$ is consonant in~$L$ if every pair of elements in~$X$ is consonant in~$L$.
\end{definition}

In particular, $L$ is \cn\ if{f} every pair in $L\times L$ is consonant.
The following lemma is a more user-friendly version of Wehrung \cite[Lemma~3.7]{MV1}, with a similar proof.

\begin{lemma}\label{L:Closat}
Let~$K$ and~$L$ be \dlat{s}, such that~$L$ has a zero, let $f\colon K\to L$ be a lattice homomorphism, and let~$a$, $a_1$, $a_2$, $b$, $b_1$, $b_2$ be elements of~$K$.
The following statements hold:
\begin{enumerater}
\item\label{j0Clos}
If~$f$ is closed at both~$(a_1,b)$ and~$(a_2,b)$, then it is closed at $(a_1\vee a_2,b)$.

\item\label{0mClos}
If~$f$ is closed at both~$(a,b_1)$ and~$(a,b_2)$, then it is closed at $(a,b_1\wedge b_2)$.

\item\label{m0Clos}
If~$(f(a_1),f(a_2))$ is consonant in~$L$ and~$f$ is closed at both~$(a_1,b)$ and~$(a_2,b)$, then~$f$ is closed at $(a_1\wedge a_2,b)$.

\item\label{0jClos}
If~$(f(b_1),f(b_2))$ is consonant in~$L$ and~$f$ is closed at both~$(a,b_1)$ and~$(a,b_2)$, then~$f$ is closed at $(a,b_1\vee b_2)$.
\end{enumerater}
\end{lemma}

\begin{proof}
\emph{Ad}~\eqref{j0Clos} and~\eqref{0mClos} are both trivial.

\emph{Ad}~\eqref{m0Clos}.
Fix a splitting pair $(y_1,y_2)$ for $(f(a_1),f(a_2))$ in~$L$ and let $x\in L$ such that $f(a_1\wedge a_2)\leq f(b)\vee x$.
Then each $f(a_i)=f(a_1\wedge a_2)\vee y_i\leq f(b)\vee x\vee y_i$, so, since~$f$ is closed at~$(a_i,b)$, there exists $c_i\in K$ such that $a_i\leq b\vee c_i$ and $f(c_i)\leq x\vee y_i$.
Setting $c\eqdef c_1\wedge c_2$, we get $a_1\wedge a_2\leq b\vee c$ whereas~$f(c)$ lies below each~$x\vee y_i$, thus below $(x\vee y_1)\wedge(x\vee y_2)=x$.

The proof of~\eqref{0jClos} is similar to the one of~\eqref{m0Clos}.
\end{proof}

The following lemma stems from Wehrung \cite[Lemma~3.9]{MV1}, with a similar proof.

\begin{lemma}\label{L:PropagClos}
Let~$K$ and~$L$ be \dlat{s} such that~$L$ has a zero, let~$\gS$ be a generating subset of~$K$, and let $f\colon K\to L$ be a lattice homomorphism such that~$f[K]$ is consonant in~$L$.
If~$f$ is closed at every pair in $\gS\times\gS$, then it is closed.
\end{lemma}

\begin{proof}
Owing to Lemma~\ref{L:Closat}, for every $a\in\gS$, the set
 \[
 X_a\eqdef\setm{b\in K}{f\text{ is closed at }(a,b)}
 \]
is a sublattice of~$K$ containing~$\gS$; thus $X_a=K$, that is, $f$ is closed at every pair in $\gS\times K$.
By the same token, it follows that for every $b\in K$, the set
 \[
 Y_b\eqdef\setm{a\in K}{f\text{ is closed at }(a,b)}
 \]
is a sublattice of~$K$ containing~$\gS$, thus equal to~$K$.
\end{proof}

\section{The lattices~$\Ops\cD$}
\label{S:OpsF}

In this section we state some required facts about the $\Ops\cD$ construction introduced in Wehrung~\cite{MV1} for the case of linear functionals and $\kk=\QQ$, and extended to arbitrary~$\kk$ and affine functionals in Wehrung~\cite{RAlg}.
Throughout this section we fix a totally ordered division ring~$\kk$.
Further, for any set~$I$ we denote by $\vecm{\gd_i}{i\in I}$ the canonical basis of the left $\kk$-vector space~$\kkp{I}$, and for $I\subseteq J$ we identify~$\kkp{I}$ with its canonical copy in~$\kkp{J}$.

\begin{notation}\label{Not:ps1}
Let~$I$ be a set and let $a\in\kkp{I}$.
We set
 \begin{equation*}
 \ps{a}\eqdef\setm{x\in\kkp{I}}{\scp{a}{x}>0}\,, \end{equation*}
whenever $a\in\kkp{I}$ (where $(a\mid x)\eqdef\sum_{i\in I}a_ix_i$).
Following Wehrung \cite{MV1,RAlg}, for any $\cD\subseteq\kkp{I}$, we shall denote by~$\Ops\cD$ the $0$-sublattice of the powerset of~$\kkp{I}$ generated by $\setm{\ps{x}}{x\in\cD}$.
\end{notation}

For a subset~$I$ of a set~$J$, we will occasionally identify $\Ops\kkp{I}$ with its canonical image in~$\Ops\kkp{J}$ (denoted $\Ops(\kkp{I},\kkp{J})$ in Plo\v{s}\v{c}ica and Wehrung \cite{MV2}).

Notation~\ref{Not:ps1} relates to the one from earlier papers such as \cite{MV2,MV1,RAlg} \emph{via} the rule $\ps{a}=\bck{a>0}$.
Our choice of the simpler notation~$\ps{a}$ is motivated by both our focus on open sets and the complexity of the equations and inequalities intervening in Sections~\ref{S:PreOOcD} and~\ref{S:ClosStep}.

We will be constantly using the obvious properties of the assignment $x\mapsto\ps{x}$ stated in the following lemma.

\begin{lemma}\label{L:Eltaryps}
The following statements hold, for all $x,y\in\kkp{I}$ and all $\gl\in\kk^+$:
\begin{enumerater}
\item\label{pscap}
$\ps{x}\cap\ps{y}\subseteq\ps{x+y}\subseteq
\ps{x}\cup\ps{y}$.

\item\label{ps0}
$\ps{x}\cap\ps{-x}=\es$.

\item\label{pshomog}
$\ps{\gl x}\subseteq\ps{x}$.
\end{enumerater}
\end{lemma}

For any subset~$C$ of~$\kkp{I}$, we denote by~$\conv(C)$ the convex hull of~$C$ and by $\cone(C)$ the additive submonoid of~$\kkp{I}$ generated by~$\kk^+C$ (so $0\in\cone(C)$).
The following lemma is implicit in Lemma~\ref{L:BakBey}.
Although it is certainly well known we could not find any reference stating it explicitly.
It can also be used to verify that \eqref{pscap}--\eqref{pshomog} from Lemma~\ref{L:Eltaryps} are, actually, defining relations of $\Ops\kkp{I}$.

\begin{lemma}\label{L:mmleqjjcoz}
Let~$A$ and~$B$ be finite subsets of~$\kkp{I}$.
The following are equivalent:
\begin{enumeratei}
\item\label{mminuu}
$\bigcap_{a\in A}\ps{a}\subseteq\bigcup_{b\in B}\ps{b}$;

\item\label{convmmcone1}
$\conv(A)\cap\cone(B)\neq\es$;

\item\label{convmmcone2}
there are elements $\xi_a,\eta_b\in\kk^+$, for $a\in A$ and $b\in B$, such that 
$\sum_{a\in A}\xi_aa=\sum_{b\in B}\eta_bb$ whereas at least one~$\xi_a$ is nonzero.
\end{enumeratei}
\end{lemma}

\begin{proof}
We may assume that~$I$ is finite.
Also, the result is trivial in case $A=\es$ (for $\conv(\es)=\es$ and, thanks to the zero element, the empty intersection is never contained in $\bigcup_{b\in B}\ps{b}$), so we may assume that $A\neq\es$.

\eqref{convmmcone2}$\Rightarrow$\eqref{mminuu}.
Owing to $\sum\nolimits_{a\in A}\xi_a>0$, for any $x\in\bigcap_{a\in A}\ps{a}$,
 \[
 0<\pI{\sum\nolimits_{a\in A}\xi_a}\min_{a\in A}\scp{a}{x}
 \leq\sum\nolimits_{a\in A}\xi_a\scp{a}{x}=
 \sum\nolimits_{b\in B}\eta_b\scp{b}{x}\,,
 \]
which, since each $\eta_b\geq0$, entails that $x\in\bigcup_{b\in B}\ps{b}$.

\eqref{mminuu}$\Rightarrow$\eqref{convmmcone1}.
Suppose that~\eqref{convmmcone1} fails, that is, $0\notin\conv(A)+(-\cone(B))$.
By the Decomposition Theorem for convex polyhedra (see for example Schrijver \cite[Corollary~7.1.b]{Schrij1986}%
\footnote{
The proof stated there for vector spaces over the reals is valid over any totally ordered division ring (in particular, it involves only [semi]linear algebra).
}),
$\conv(A)+(-\cone(B))$ is a convex polyhedron, that is, a finite intersection of closed affine half-spaces, of~$\kkp{I}$.
Since $0\notin\conv(A)+(-\cone(B))$, there must exist $x\in\kkp{I}$ such that $\scp{c}{x}\geq1$ whenever $c\in\conv(A)+(-\cone(B))$.
Since $0\in\cone(B)$, it follows that $\scp{a}{x}>0$ whenever $a\in A$; so $x\in\bigcap_{a\in A}\ps{a}$.
Moreover, for all $a\in A$, $b\in B$, and $\gl\in\kk^+$, $\scp{a-\gl b}{x}>0$, that is, $\scp{a}{x}>\gl\scp{b}{x}$.
Taking~$\gl$ large enough yields $\scp{b}{x}\le 0$, whence $x\in\bigcap_{a\in A}\ps{a}\setminus\bigcup_{b\in B}\ps{b}$; that is, \eqref{mminuu} fails.

\eqref{convmmcone1}$\Rightarrow$\eqref{convmmcone2} is trivial.
\end{proof}

\begin{remark}\label{Rk:mmleqjjcoz}
The case $B=\es$ in Lemma~\ref{L:mmleqjjcoz} states that $\bigcap_{a\in A}\ps{a}=\es$ if{f} $0\in\conv A$.
\end{remark}

%

Denote by~$\FL(I,\kk)$ the free $\kk$-vector lattice on a set~$I$.
As observed in Baker~\cite{Baker1968}, Bernau~\cite{Bern1969}, Madden \cite[Ch.~III]{MaddTh} (see also Wehrung \cite[page~13]{RAlg} for a summary), $\FL(I,\kk)$ canonically embeds into~$\kk^{\kkp{I}}$.
We sum up the corresponding facts that are relevant to our discussion.

\goodbreak
\begin{lemma}[Folklore]\label{L:BakBey}\hfill
\begin{enumerater}
\item\label{FlIkk}
$\FL(I,\kk)$ is isomorphic to the sublattice of $\kk^{\kkp{I}}$ generated by all linear functionals $\sum_{i\in I}a_i\gd_i$ associated to elements $a\in\kkp{I}$, via the assignment $i\mapsto\gd_i$.

\item\label{IdcFlIkk}
The assignment $\seq{x^+}\mapsto\ps{x}$ defines a lattice isomorphism $\Idc\FL(I,\kk)\to\Ops\kkp{I}$.
\end{enumerater}
\end{lemma}

\section{Homomorphisms on the lattices~$\Ops\cD$}
\label{S:PreOOcD}

\emph{Standing hypothesis}: $\kk$ is a totally ordered division ring, $J=I\sqcup\set{o}$ is a set, and~$L$ is a \dzlat.

A set~$\cD$ of vectors is \emph{symmetric} if $-u\in\cD$ whenever $u\in\cD$.
An element $a\in\kkp{J}$ is \emph{normalized} if $a_o\in\set{-1,0,1}$, and a subset~$\cD$ of~$\kkp{J}$ is normalized if each of its elements is normalized.

A map $\gf\colon\Ops\cD\to L$ \emph{extends} a map $f\colon\cD\to L$ if $\gf\ps{u}=f(u)$ whenever $u\in\cD$ (in order to avoid cluttering we are writing~$\gf\ps{u}$ instead of~$\gf(\ps{u})$).

\begin{lemma}\label{L:ExtIcupD}
Let~$\cD$ be a normalized symmetric subset of $\kkp{J}\setminus\kkp{I}$, let $c\in\kkp{J}$ be normalized, and let $\bc^+,\bc^-\in\nobreak L$.
A $0$-lattice homomorphism $\gf\colon\Ops(\kkp{I}\cup\cD)\to L$ extends to some $0$-lattice homomorphism $\psi\colon\Ops(\kkp{I}\cup\cD\cup\set{c,-c})\to\nobreak L$ such that $(\psi\ps{c},\psi\ps{-c})=(\bc^+,\bc^-)$ if{f} the following inequalities hold for any $u\in\cD\cap(c+\kkp{I})$:
 \begin{align}
 \bc^+\wedge\bc^-&=0\,;\label{Eq:c+-=0}\\
 \bc^+&\leq\gf\ps{c-u}\vee\gf\ps{u}\,;\label{Eq:UBa+}\\
 \gf\ps{u}&\leq\gf\ps{u-c}\vee\bc^+\,;\label{Eq:LB1a+}\\
 \gf\ps{c-u}&\leq\gf\ps{-u}\vee\bc^+\,;\label{Eq:LB2a+}\\
 \bc^-&\leq\gf\ps{u-c}\vee\gf\ps{-u}\,;\label{Eq:UBa-}\\
 \gf\ps{-u}&\leq\gf\ps{c-u}\vee\bc^-\,;\label{Eq:LB1a-}\\
 \gf\ps{u-c}&\leq\gf\ps{u}\vee\bc^-\,.\label{Eq:LB2a-}
 \end{align}
Furthermore, if~$\cD$ is finite and the range of~$\gf$ is consonant in~$L$, then such a pair $(\bc^+,\bc^-)$ always exists.
\end{lemma}

\begin{proof}
The given conditions are obviously necessary (cf. Lemma~\ref{L:Eltaryps}).
This also yields the desired equivalence in case $c\in\kkp{I}\cup\cD$ (in which case $\psi=\gf$).

Thus suppose, from now on, that $c\notin\kkp{I}\cup\cD$ and that \eqref{Eq:c+-=0}--\eqref{Eq:LB2a-} all hold.
We need to verify that there exists an extension~$\psi$ of~$\gf$ as stated.
Set $\cD'\eqdef\cD\cup\set{c,-c}$ and define a map $g\colon\kkp{I}\cup\cD'\to L$ by setting $g(x)\eqdef\gf\ps{x}$ whenever $x\in\kkp{I}\cup\cD$, whereas $g(c)\eqdef\bc^+$ and $g(-c)\eqdef\bc^-$.
It suffices to verify that~$g$ extends to a $0$-lattice homomorphism $\Ops(\kkp{I}\cup\cD')\to L$ (for any such homomorphism would necessarily extend~$\gf$).
By Plo\v{s}\v{c}ica and Wehrung \cite[Lemma~4.4]{MV2}, it suffices to verify that for any integers $k>0$ and $l\geq0$, together with elements $u_p,v_q\in\kkp{I}\cup\cD'$ for $1\leq p\leq k$ and $1\leq q\leq l$, the containment
 \begin{equation}\label{Eq:ContkkpID'}
 \bigcap\nolimits_{1\leq p\leq k}\ps{u_p}
 \subseteq
 \bigcup\nolimits_{1\leq q\leq l}\ps{v_q}
 \end{equation}
entails the inequality
 \begin{equation}\label{Eq:IneqkkpD'}
 \bigwedge\nolimits_{1\leq p\leq k}g(u_p)
 \leq\bigvee\nolimits_{1\leq q\leq l}g(v_q)\,.
 \end{equation}
Set $U\eqdef\set{u_1,\dots,u_k}$ and $V\eqdef\set{v_1,\dots,v_l}$.
We argue by induction on the lexicographically ordered pair $(k+l,m)$ where~$m$ is the cardinality of $(U\cup V)\setminus\kkp{I}$.
We may thus assume that no proper subfamily of $(u_1,\dots,u_k,v_1,\dots,v_l)$ satisfies~\eqref{Eq:ContkkpID'};
we will express this by stating that ``$k+l$ is minimal subjected to~\eqref{Eq:IneqkkpD'}''.
Because of this and since~\eqref{Eq:IneqkkpD'} trivially holds in case some $u_p=v_q$, we may also assume that~$u_1$, \dots, $u_k$, $v_1$, \dots, $v_l$ are pairwise distinct.

\begin{sclaim}
For all $x,y\in\cD'$, $x+y\in\kkp{I}$ implies that $g(x)\wedge g(y)\leq g(x+y)\leq g(x)\vee g(y)$ and $g(x)\leq g(x+y)\vee g(-y)$.
\end{sclaim}

\begin{scproof}
The inequalities $g(x+y)\leq g(x)\vee g(y)$ and $g(x)\leq g(x+y)\vee g(-y)$ follow from~$\gf$ being a \jh, if $\set{x,y}\subseteq\cD$, and from \eqref{Eq:c+-=0}--\eqref{Eq:LB2a-}, if $\set{x,y}\cap\set{c,-c}\neq\es$.
Since $g(y)\wedge g(-y)=0$ (use our assumption~\eqref{Eq:c+-=0}), the second inequality implies that $g(x)\wedge g(y)\leq g(x+y)$.
\end{scproof}

It follows from Lemma~\ref{L:mmleqjjcoz} that there are $\ga_p,\gb_q\in\kk^+$, for $1\leq p\leq k$ and $1\leq q\leq l$, such that some~$\ga_p$ is nonzero and
 \begin{equation}\label{Eq:aupbvq}
 \sum\nolimits_{1\leq p\leq k}\ga_pu_p=
 \sum\nolimits_{1\leq q\leq l}\gb_qv_q\,.
 \end{equation}
Since $k+l$ is minimal subjected to~\eqref{Eq:IneqkkpD'}, all~$\ga_p$ and all~$\gb_q$ are nonzero.
If $m=0$ then $U\cup V\subseteq\kkp{I}$ and the conclusion~\eqref{Eq:IneqkkpD'} follows from~$\gf$ being a lattice homomorphism.
If~$m$ is nonzero, we separate cases.

\setcounter{case}{0}

\begin{case}\label{Ca:v1+v2}
There are distinct~$p$, $q$ such that $\set{v_p,v_q}\subseteq\cD'$ and $v_p+v_q\in\kkp{I}$.
\end{case}

Up to permutation of the~$v_r$, we may assume that $\set{v_1,v_2}\subseteq\cD'$, $v_1+v_2\in\kkp{I}$, and $\gb_1\leq\gb_2$.
Now~\eqref{Eq:aupbvq} can be rewritten as
 \[
 \sum\nolimits_{1\leq p\leq k}\ga_pu_p
 =\gb_1(v_1+v_2)+(\gb_2-\gb_1)v_2+
 \sum\nolimits_{3\leq q\leq l}\gb_qv_q\,.
 \]
It follows that $\bigcap_{1\leq p\leq k}\ps{u_p}\subseteq \ps{v_1+v_2}\cup\bigcup_{2\leq q\leq l}\ps{v_q}$.
By our induction hypothesis (which applies since $v_1\in\cD'$ whereas $v_1+v_2\in\kkp{I}$), we get $\bigwedge_{1\leq p\leq k}g(u_p)\leq g(v_1+v_2)\vee\bigvee_{2\leq q\leq l}g(v_q)$, which, by the Claim above applied to the inequality $g(v_1+v_2)\leq g(v_1)\vee g(v_2)$, entails~\eqref{Eq:IneqkkpD'}.

\begin{case}\label{Ca:u1+u2}
There are distinct~$p$, $q$ such that $\set{u_p,u_q}\subseteq\cD'$ and $u_p+u_q\in\kkp{I}$.
\end{case}

Up to permutation of the~$u_r$, we may assume that $\set{u_1,u_2}\subseteq\cD'$, $u_1+u_2\in\kkp{I}$, and $\ga_1\leq\ga_2$.
Now~\eqref{Eq:aupbvq} can be rewritten as
 \[
 \ga_1(u_1+u_2)+(\ga_2-\ga_1)u_2+\sum\nolimits_{3\leq p\leq k}\ga_pu_p=
 \sum\nolimits_{1\leq q\leq l}\gb_qv_q\,.
 \]
Since at least one element of $\set{\ga_1,\ga_2-\ga_1,\ga_3,\dots,\ga_k}$ is positive, it follows from our induction hypothesis that
 \[
 g(u_1+u_2)\wedge
 \bigwedge\nolimits_{2\leq p\leq k}g(u_k)\leq
 \bigvee\nolimits_{1\leq q\leq l}g(v_q)\,.
 \]
Since, by our Claim, $g(u_1)\wedge g(u_2)\leq g(u_1+u_2)$, \eqref{Eq:IneqkkpD'} follows.

Now suppose that neither Case~\ref{Ca:v1+v2} nor Case~\ref{Ca:u1+u2} occurs.
Since~$\cD'$ is normalized, all~$u_p$ from~$\cD'$ have the same $o$-coordinate~$\eps$ and all~$v_q$ from~$\cD'$ have the same $o$-coordinate~$\eta$ (so $\set{\eps,\eta}\subseteq\set{-1,1}$).
Since $m>0$, there must be at least one such vector on each side, and further, due to~\eqref{Eq:aupbvq}, we get $\eps=\eta$.
We may thus assume that $\set{u_1,v_1}\subseteq\cD'$ and $u_1-v_1\in\kkp{I}$.
Two cases may then occur.

\begin{case}\label{Ca:a1>b1}
$\ga_1>\gb_1$.
\end{case}

Rewriting~\eqref{Eq:aupbvq} as
 \[
 (\ga_1-\gb_1)u_1+
 \sum\nolimits_{2\leq p\leq k}\ga_pu_p=
 \gb_1(v_1-u_1)+
 \sum\nolimits_{2\leq q\leq l}\gb_qv_q\,,
 \]
our induction hypothesis entails
 \begin{equation}\label{Eq:gv1-u1leq}
 \bigwedge\nolimits_{1\leq p\leq k}g(u_p)\leq
 g(v_1-u_1)\vee
 \bigvee\nolimits_{2\leq q\leq l}g(v_q)\,.
 \end{equation}
Now our Claim entails $g(v_1-u_1)\leq g(v_1)\vee g(-u_1)$, thus, since the left hand side of~\eqref{Eq:gv1-u1leq} lies below~$g(u_1)$ and $g(u_1)\wedge g(-u_1)=0$, \eqref{Eq:IneqkkpD'} holds.

\begin{case}\label{Ca:a1leqb1}
$\ga_1\leq\gb_1$.
\end{case}

Rewriting~\eqref{Eq:aupbvq} as
 \[
 \ga_1(u_1-v_1)+
 \sum\nolimits_{2\leq p\leq k}\ga_pu_p=
 (\gb_1-\ga_1)v_1+
 \sum\nolimits_{2\leq q\leq l}\gb_qv_q\,,
 \]
our induction hypothesis entails
 \[
 g(u_1-v_1)\wedge
 \bigwedge\nolimits_{2\leq p\leq k}g(u_p)\leq
 \bigvee\nolimits_{1\leq q\leq l}g(v_q)\,.
 \]
Since, by our Claim, $g(u_1)\leq g(u_1-v_1)\vee g(v_1)$, \eqref{Eq:IneqkkpD'} follows.
This concludes the proof of the equivalence between the existence of~$\psi$ and the system of inequalities \eqref{Eq:c+-=0}--\eqref{Eq:LB2a-}.

Now suppose that~$\cD$ is finite and the range of~$\gf$ is consonant in~$L$.
Since the set $\cK\eqdef\cD\cup\pI{\kkp{I}\cap((c+\cD)\cup(-c+\cD))}$ is finite, it follows from Wehrung \cite[Lemma~6.6]{MV1} (stated there on Abelian \lgrp{s}; the $\kk$-vector lattice version, stated in Wehrung \cite[Lemma~4.7]{RAlg}, is, \emph{mutatis mutandis}, identical) that the restriction of~$\gf$ to~$\Ops\cK$ extends to a $0$-lattice homomorphism $\psi_0\colon\Ops(\cK\cup\set{c,-c})\to L$.
Since $\set{u,-u,u-c,c-u}$ is contained in~$\cK$ whenever $u\in\cD\cap(c+\kkp{I})$, all conditions \eqref{Eq:c+-=0}--\eqref{Eq:LB2a-}, with $\bc^+\eqdef\psi_0\ps{c}$ and $\bc^-\eqdef\psi_0\ps{-c}$, are satisfied.
\end{proof}

\section{The Closure Step for open half-spaces}
\label{S:ClosStep}

\emph{Standing hypothesis}: $\kk$ is a totally ordered division ring and~$L$ is a \dzlat.

The following preparatory lemma aims at providing a better understanding of a map be closed at a pair of open half-spaces.

\begin{lemma}\label{L:Closatab}
The following statements hold, for every $0$-lattice homomorphism $\gf\colon\Ops\kkp{I}\to L$ and all $a,b\in\kkp{I}$:
\begin{enumerater}
\item\label{ConjDisjClos}
For every $\gl\in\kk^+$ and every $\be\in L$, the inequality $\gf\ps{a}\wedge \gf\ps{a-\gl b}\leq\be$ is equivalent to the conjunction of $\gf\ps{a}\leq \gf\ps{b}\vee\be$ and $\gf\ps{a-\gl b}\leq \gf\ps{-b}\vee\be$.

\item\label{IncrClos}
Suppose that $\gf\ps{a}\leq \gf\ps{b}\vee\be$.
Then for all $\gl,\gl'\in\kk$ such that $0\leq\gl\leq\gl'$, $\gf\ps{a}\wedge \gf\ps{a-\gl b}\leq\be$ implies that $\gf\ps{a}\wedge \gf\ps{a-\gl'b}\leq\be$.

\item\label{HypClos}
The map~$\gf$ is closed at $(\ps{a},\ps{b})$ if{f} for every $\be\in L$ such that $\gf\ps{a}\leq \gf\ps{b}\vee\be$ there exists $\gl\in\kk^+$ such that $\gf\ps{a-\gl b}\leq \gf\ps{-b}\vee\be$.
\end{enumerater}
\end{lemma}

\begin{proof}
\emph{Ad}~\eqref{ConjDisjClos}.
Suppose first that $\gf\ps{a}\wedge\gf\ps{a-\gl b}\leq\be$.
{F}rom $\gl\geq0$ it follows that $\ps{a}\subseteq\ps{b}\cup(\ps{a}\cap\ps{a-\gl b})$, thus
 \[
 \gf\ps{a}\leq \gf\ps{b}\vee\pI{\gf\ps{a}\wedge \gf\ps{a-\gl b}}
 \leq \gf\ps{b}\vee\be\,. 
 \]
Moreover, from $\gl\geq0$ it follows that $\ps{a-\gl b}\subseteq\ps{a}\cup\ps{-b}$, so we get
 \[
 \gf\ps{a-\gl b}\leq\pI{\gf\ps{a}\wedge \gf\ps{a-\gl b}}\vee
 \gf\ps{-b}\leq \gf\ps{-b}\vee\be\,.
 \]
Suppose, conversely, that $\gf\ps{a}\leq \gf\ps{b}\vee\be$ and $\gf\ps{a-\gl b}\leq \gf\ps{-b}\vee\be$.
{F}rom the former inequality together with $\gf\ps{b}\wedge \gf\ps{-b}=0$ it follows that $\gf\ps{a}\wedge \gf\ps{-b}\leq\be$, whence
 \[
 \gf\ps{a}\wedge \gf\ps{a-\gl b}\leq
 \gf\ps{a}\wedge(\gf\ps{-b}\vee\be)\leq\be\,.
 \]

\emph{Ad}~\eqref{IncrClos}.
Suppose that $\gf\ps{a}\wedge \gf\ps{a-\gl b}\leq\be$.
By~\eqref{ConjDisjClos} above, this means that $\gf\ps{a-\gl b}\leq \gf\ps{-b}\vee\be$.
Hence $\gf\ps{a-\gl'b}\leq \gf\ps{a-\gl b}\vee \gf\ps{(\gl-\gl')b}\leq \gf\ps{-b}\vee\be$, so, by~\eqref{ConjDisjClos} above, $\gf\ps{a}\wedge \gf\ps{a-\gl'b}\leq\be$.

\emph{Ad}~\eqref{HypClos}.
Suppose first that~$\gf$ is closed at $(\ps{a},\ps{b})$.
Identifying the lattices $\Ops\kkp{I}$ and~$\Idc\FL(I,\kk)$ \emph{via} the isomorphism $\ps{x}\mapsto\seq{x^+}$ given by Lemma~\ref{L:BakBey}, any inequality of the form $\gf\ps{a}\leq \gf\ps{b}\vee\be$, where $\be\in L$, translates to $\gf\seq{a^+}\leq\gf\seq{b^+}\vee\be$, thus, since~$\gf$ is closed at $(\ps{a},\ps{b})$, there exists $e\in\FL(I,\kk)^+$ such that $\seq{a^+}\leq\seq{b^+}\vee\seq{e}$ whereas $\gf\seq{e}\leq\be$.
The former inequality means that $a^+\leq\gl(b^++e)$ for some $\gl\in\kk^+$, so $(a^+-\gl b^+)^+\leq\gl e$.
By virtue of the identity $(a^+-\gl b^+)^+=a^+\wedge(a-\gl b)^+$ we thus get
 \[
 \gf\seq{a^+}\wedge\gf\seq{(a-\gl b)^+}\leq\gf\seq{\gl e}\leq
 \gf\seq{e}\leq\be\,,
 \]
which, owing to the identification given by $\Ops\kkp{I}\cong\Idc\FL(I,\kk)$\,, can be written
 \[
 \gf\ps{a}\wedge \gf\ps{a-\gl b}\leq\be\,.
 \]
Conversely, if for every $\be\in L$ such that $\gf\ps{a}\leq \gf\ps{b}\vee\be$ there exists $\gl\in\kk^+$ such that $\gf\ps{a-\gl b}\leq \gf\ps{-b}\vee\be$, that is, due to~\eqref{ConjDisjClos} above, $\gf(\ps{a}\cap\ps{a-\gl b})\leq\be$, then, due to the containment 
$\ps{a}\subseteq\ps{b}\cup(\ps{a}\cap\ps{a-\gl b})$, $\gf$ is closed at $(\ps{a},\ps{b})$.
\end{proof}

For the remainder of this section let $J=I\cup\set{o}$ with $o\notin I$, let~$\cD$ be a symmetric, normalized subset of~$\kkp{J}\setminus\kkp{I}$, let $\set{a,b}\subseteq\kkp{I}\cup\cD$, let $\be\in L$, and let\newline $\gf\colon\Ops(\kkp{I}\cup\cD)\to L$ be a \jzh.
We also assume that the restriction of~$\gf$ to~$\Ops\kkp{I}$ is closed and that $\gf\ps{a}\leq\gf\ps{b}\vee\be$.

\begin{notation}\label{Not:x^u}
For all $x\in\kkp{J}$ and $u\in\kkp{J}\setminus\kkp{I}$, we denote by~$x^u$ the unique element in $\kkp{I}\cap(x+\kk u)$.
That is, $x^u=x-x_ou_o^{-1}u$.
\end{notation}

Observe that $x^u=x^{-u}$.
Also, in all the cases that we will consider, $u$ will be normalized (i.e., $u_o\in\set{-1,1}$); so $u_o^{-1}=u_o$.

\begin{lemma}\label{L:Old31}
Assume $\set{a,b}\not\subseteq\kkp{I}$ and let $u\in\cD$ such that either $b_ou_o=1$ or \pup{$b_o=0$ and $a_ou_o=-1$}.
The following statements hold:
\begin{enumerater}
\item\label{Initaulbu}
$\gf\ps{a^u}\leq\gf\ps{b^u}\vee\gf\ps{u}\vee\gf\ps{-b}\vee\be$.

\item\label{aulbu}
For all large enough $\gl\in\kk^+$,
$\gf\ps{a^u-\gl b^u}\leq\gf\ps{u}\vee\gf\ps{-b}\vee\be$.

\end{enumerater}
\end{lemma}

\begin{proof}
\emph{Ad}~\eqref{Initaulbu}.
{F}rom $a^u=a-a_ou_ou$ and $\gf\ps{a}\leq\gf\ps{b}\vee\be$ it follows that
 \begin{equation}\label{Eq:Inita^uaouo}
 \gf\ps{a^u}\leq\gf\ps{a}\vee\gf\ps{-a_ou_ou}\leq
 \gf\ps{b}\vee\gf\ps{-a_ou_ou}\vee\be\,.
 \end{equation}
If $b_o=0$ and $a_ou_o=-1$, then $b^u=b$ and~\eqref{Initaulbu} follows.
Let $b_ou_o=1$; so $b^u=b-u$.
Then $\gf\ps{b}\leq\gf\ps{b^u}\vee\gf\ps{u}$ thus, by~\eqref{Eq:Inita^uaouo},
 \begin{equation}\label{Eq:Inita^uaouo2}
 \gf\ps{a^u}\leq\gf\ps{b^u}\vee\gf\ps{u}\vee
 \gf\ps{-a_ou_ou}\vee\be\,.
 \end{equation}
Hence, if $a_ou_o\in\set{-1,0}$ then we get~\eqref{Initaulbu} right away.
If $a_ou_o=1$, then
 \[
 \gf\ps{-a_ou_ou}=\gf\ps{-u}\leq\gf\ps{b^u}\vee\gf\ps{-b}\,,
 \]
which, combined with~\eqref{Eq:Inita^uaouo2}, yields again~\eqref{Initaulbu}.

\emph{Ad}~\eqref{aulbu}.
Since~$\gf$ is closed at $(a^u,b^u)$ and by~\eqref{Initaulbu} together with Lemma~\ref{L:Closatab}, there exists $\gl_0\in\kk^+$ such that 
 \begin{equation}\label{Eq:gfclaubu}
 \gf\ps{a^u-\gl b^u}\leq
 \gf\ps{-b^u}\vee\gf\ps{u}\vee\gf\ps{-b}\vee\be\quad
 \text{whenever }\gl>\gl_0\,.
 \end{equation}
If $b_o=0$, then $b^u=b$ and we get~\eqref{aulbu} right away.
If $b_ou_o=1$ then $b^u=b-u$, thus 
$\gf\ps{-b^u}\leq\gf\ps{u}\vee\gf\ps{-b}$, which, together with~\eqref{Eq:gfclaubu}, yields~\eqref{aulbu} again.
\end{proof}

\begin{lemma}\label{L:Old32}
Assume $\set{a,b}\not\subseteq\kkp{I}$ and suppose that either $b_ou_o=-1$ or \pup{$b_o=0$ and $a_ou_o=1$}.
Then
 \begin{equation}\label{Eq:Minglbu-au}
 \gf\ps{u}\leq\gf\ps{\gl b^u-a^u}\vee\gf\ps{-b}\vee\be\quad
 \text{for all large enough }\gl\in\kk^+\,.
 \end{equation}
\end{lemma}

\begin{proof}
Since $-u$ satisfies the assumptions of Lemma~\ref{L:Old31}, there is $\gl_0\in\kk^+$ such that
 \begin{equation}\label{Eq:Majau-glbu}
 \gf\ps{a^u-\gl b^u}\leq\gf\ps{-u}\vee\gf\ps{-b}\vee\be
 \quad\text{whenever }\gl>\gl_0\,.
 \end{equation}
Let $\gl>\gl_0$; set $\be'\eqdef\gf\ps{\gl b^u-a^u}\vee\gf\ps{-b}\vee\be$ and $\be''\eqdef\be'\vee\gf\ps{-u}$.

We claim that $\gf\ps{b^u}\leq\be''$.
Pick $\gl'\in\kk$ with $\gl_0<\gl'<\gl$.
{F}rom the containment $\ps{b^u}\subseteq\ps{\gl b^u-a^u}\cup\ps{a^u-\gl'b^u}$ it follows that $\gf\ps{b^u}\leq\gf\ps{\gl b^u-a^u}\vee\gf\ps{a^u-\gl'b^u}$, which, by~\eqref{Eq:Majau-glbu} together with the definition of~$\be''$, entails $\gf\ps{b^u}\leq\be''$, as claimed.

We next claim that $\gf\ps{u}\leq\be''$.
If $b_ou_o=-1$ then $b^u=b+u$, thus
 \[
 \gf\ps{u}\leq\gf\ps{b^u}\vee\gf\ps{-b}\leq\be''\,,
 \]
as desired.
Suppose now that $b_o=0$; so $b^u=b$ and $a^u=a-u$.
{F}rom the containment $\ps{-a^u}\subseteq\ps{\gl b^u-a^u}\cup\ps{-b}$ we get $\gf\ps{-a^u}\leq\gf\ps{\gl b^u-a^u}\vee\gf\ps{-b}\leq\be'$, thus, since $\ps{u}\subseteq\ps{a}\cup\ps{-a^u}$, we get
 \begin{align*}
 \gf\ps{u}&\leq\gf\ps{a}\vee\gf\ps{-a^u}\\
 &\leq\gf\ps{a}\vee\be'\\
 &\leq\gf\ps{b}\vee\be'
 &&(\text{because }\gf\ps{a}\leq\gf\ps{b}\vee\be)\\
 &=\gf\ps{b^u}\vee\be'
 &&(\text{because }b^u=b)\\
 &\leq\be''\,, 
 \end{align*}
thus completing the proof of our second claim.
Since $\gf\ps{u}\leq\be''=\be'\vee\gf\ps{-u}$ and $\gf\ps{u}\wedge\gf\ps{-u}=0$, it follows that $\gf\ps{u}\leq\be'$.
\end{proof}

\begin{lemma}\label{L:ClStep}
Suppose that~$\cD$ is finite and the range of~$\gf$ is consonant in~$L$.
Then for all large enough $\gl\in\kk^+$, the map~$\gf$ extends to a lattice homomorphism $\psi\colon\Ops(\kkp{I}\cup\cD\cup\set{a-\gl b,\gl b-a})\to L$ such that
$\psi\ps{a-\gl b}\leq\gf\ps{-b}\vee\be$.
\end{lemma}

\begin{proof}
If $\set{a,b}\subseteq\kkp{I}$ then the desired conclusion follows from the restriction of~$\gf$ to~$\Ops\kkp{I}$ being closed (cf. Lemma~\ref{L:Closatab}).

Suppose from now on that $\set{a,b}\not\subseteq\kkp{I}$.
Let~$\gl_0\in\kk^+$ such that all inequalities in Lemmas~\ref{L:Old31} and~\ref{L:Old32} are satisfied whenever $u\in\cD$.
We may also assume that $a_o-\gl b_o$ has constant sign, necessarily nonzero, over $\gl>\gl_0$.
Fixing such a~$\gl$, the scalar $\xi\eqdef|a_o-\gl b_o|^{-1}$ is positive, and further, $c\eqdef\xi\cdot(a-\gl b)$ is a normalized element of $\kkp{J}\setminus\kkp{I}$.
For any $u\in\cD$, two cases may occur.

\setcounter{case}{0}

\begin{case}\label{Ca:uo+co=0}
$u_o+c_o=0$ \pup{i.e., $u+c\in\kkp{I}$}.
\end{case}

It follows that either $b_ou_o=1$ or ($b_o=0$ and $a_ou_o=-1$).
Then $u+c=c^u=\xi\cdot(a^u-\gl b^u)$, whence, applying Lemma~\ref{L:Old31},
 \begin{equation}\label{Eq:1/2}
 \gf\ps{u+c}=\gf\ps{a^u-\gl b^u}\leq
 \gf\ps{u}\vee\gf\ps{-b}\vee\be\,.
 \end{equation}

\begin{case}\label{Ca:uo=co}
$u_o=c_o$ \pup{i.e., $u-c\in\kkp{I}$}.
\end{case}

It follows that either $b_ou_o=-1$ or ($b_o=0$ and $a_ou_o=1$).
Then $u-c=-c^u=\xi\cdot(\gl b^u-a^u)$, whence $\gf\ps{u-c}=\gf\ps{\gl b^u-a^u}$, and so, by Lemma~\ref{L:Old32},
 \begin{equation}\label{Eq:2/2}
 \gf\ps{u}\leq\gf\ps{\gl b^u-a^u}\vee\gf\ps{-b}\vee\be
 =\gf\ps{u-c}\vee\gf\ps{-b}\vee\be\,.
 \end{equation}
Since~$\cD$ is finite and the range of~$\gf$ is consonant in~$L$, it follows from the last statement of Lemma~\ref{L:ExtIcupD} that the system of inequalities \eqref{Eq:c+-=0}--\eqref{Eq:LB2a-} has a solution $(\bc^+,\bc^-)$.
Set $\bc^*\eqdef\bc^+\wedge(\gf\ps{-b}\vee\be)$.
We claim that the pair $(\bc^*,\bc^-)$ also satisfies the system of inequalities \eqref{Eq:c+-=0}--\eqref{Eq:LB2a-}: indeed, since $(\bc^+,\bc^-)$ already satisfies that system and $\bc^*\leq\bc^+$, the only inequalities that need to be taken care of are~\eqref{Eq:LB1a+} and~\eqref{Eq:LB2a+}, which then follow from~\eqref{Eq:1/2} and~\eqref{Eq:2/2} above.

By Lemma~\ref{L:ExtIcupD}, the homomorphism~$\gf$ extends to a lattice homomorphism\linebreak $\psi\colon\Ops(\kkp{I}\cup\cD\cup\set{c,-c})\to L$ such that $\psi\ps{c}=\bc^*$ and $\psi\ps{-c}=\bc^-$.
In particular, $\psi\ps{a-\gl b}=\psi\ps{c}\leq\gf\ps{-b}\vee\be$.
\end{proof}

\section{Proof of the main result}\label{S:Main}

\begin{theorem}[The Closure Step]\label{T:ClosStep}
Let~$J\eqdef I\sqcup\set{o}$ be a countable set, let~$\kk$ be a countable totally ordered division ring, let~$L$ be a countable \cn\ \dzlat, and let $\bd_0,\bd_1\in L$ such that $\bd_0\wedge\bd_1=0$.
Then every closed $0$-lattice homomorphism $\gf\colon\Ops\kkp{I}\to L$ extends to a closed $0$-lattice homomorphism $\psi\colon\Ops\kkp{J}\to L$ such that $\psi\ps{\gd_o}=\bd_0$ and $\psi\ps{-\gd_o}=\bd_1$.
\end{theorem}

\begin{proof}
Denote by~$\cN$ the set of all normalized elements of~$\kkp{J}\setminus\kkp{I}$.
Since~$\kkp{J}$ and~$L$ are both countable, we can write $\cN=\setm{c_n}{n<\go}$ and $L=\setm{\be_n}{n<\go}$.

Define inductively finite symmetric subsets~$\cD_n$ of~$\cN$ and $0$-lattice homomorphisms $\gf_n\colon\Ops(\kkp{I}\cup\cD_n)\to L$, for $n<\go$, as follows.
Set $\cD_0\eqdef\set{\gd_o,-\gd_o}$.
By the argument of Wehrung \cite[Lemma~8.3]{MV1} (alternatively, apply Lemma~\ref{L:ExtIcupD} with $\cD\eqdef\es$), there is a unique $0$-lattice homomorphism $\gf_0\colon\Ops(\kkp{I}\cup\cD_0)\to L$ such that $\gf_0\ps{\gd_o}=\bd_0$ and $\gf_0\ps{-\gd_o}=\bd_1$.

Suppose~$\cD_n$ and~$\gf_n$ defined.

If~$n$ is even, set $\cD_{n+1}\eqdef\cD_n\cup\set{c_{n/2},-c_{n/2}}$.
By the last statement of Lemma~\ref{L:ExtIcupD}, $\gf_n$ extends to a lattice homomorphism $\gf_{n+1}\colon\Ops(\kkp{I}\cup\cD_{n+1})\to L$.

If~$n$ is odd, then a straightforward finite iteration of Lemma~\ref{L:ClStep} yields a finite symmetric subset~$\cD_{n+1}$ of~$\cN$ containing~$\cD_n$ and a $0$-lattice homomorphism $\gf_{n+1}\colon\Ops(\kkp{I}\cup\cD_{n+1})\to L$ such that for all $a,b\in\cD_n$ and all $k\leq n$,
 \begin{equation}\label{Eq:LocalClosf_n}
 \gf_n\ps{a}\leq\gf_n\ps{b}\vee\be_k\Rightarrow
 (\exists\gl\in\kk^+)
 (a-\gl b\in\cD_{n+1}\text{ and }
\gf_{n+1}\ps{a-\gl b}\leq\gf_n\ps{-b}\vee\be_k)\,.
 \end{equation}
By construction, $\cN=\bigcup_{n<\go}\cD_n$ and the union $\psi\eqdef\bigcup_{n<\go}\gf_n$ is a $0$-lattice homomorphism from $\Ops(\kkp{I}\cup\cN)=\Ops\kkp{J}$ to~$L$.
Moreover, owing to~\eqref{Eq:LocalClosf_n} together with Lemma~\ref{L:Closatab}, $\psi$ is closed at every pair $(\ps{a},\ps{b})$ where $a,b\in\kkp{J}$.
Since $\setm{\ps{x}}{x\in\kkp{J}}$ generates~$\Ops\kkp{J}$ as a sublattice, it follows from Lemma~\ref{L:PropagClos} that~$\psi$ is closed.
\end{proof}

\begin{theorem}\label{T:ClosStep2}
Let~$J=I\sqcup\omega$ be a countable set, let~$\kk$ be a countable totally ordered division ring, and  let~$L$ be a countable \cn\ \dzlat.
Then every closed $0$-lattice homomorphism $\psi_0\colon\Ops\kkp{I}\to L$ extends to a surjective closed $0$-lattice homomorphism $\psi\colon\Ops\kkp{J}\to L$.
\end{theorem}

\begin{proof} 
Set $I_n\eqdef I\cup\set{0,1,\dots,n-1}$, whenever $n<\go$.
Moreover, let $L=\setm{\be_n}{n<\go}$.
By a repeated use of Theorem \ref{T:ClosStep}
we obtain a sequence of closed $0$-lattice homomorphisms $\psi_n\colon\Ops\kkp{I_n}\to L$, each~$\psi_n$ extending the previous one, and satisfying
 $\psi_{n+1}\ps{\gd_n}=\be_n$ and $\psi_{n+1}\ps{-\gd_n}=0$.
Define~$\psi$ as the common extension of all $\psi_n$. 
\end{proof}

\begin{theorem}[Main Theorem]\label{T:Main}
Let~$\kk$ be a countable totally ordered division ring, let~$E$ be a countable $\kk$-vector lattice, let~$L$ be a countable \cn\ \dzlat, and let $\gf\colon\Idc{E}\to L$ be a closed $0$-lattice homomorphism.
Then there are a countable $\kk$-vector lattice~$F$, a $\kk$,\lhom\ $f\colon E\to F$, and a lattice isomorphism $\gi\colon\Idc{F}\to L$ such that $\gf=\gi\circ\Idc{f}$.
\end{theorem}

\begin{proof}
Throughout the proof let us fix a countably infinite set~
$I$, disjoint from~$\go$, and set $J\eqdef I\sqcup\go$.

Since~$E$ is countable, there exists a surjective $\kk$,\lhom\ $p\colon\FL(I,\kk)\twoheadrightarrow\nobreak E$.
By Wehrung \cite[Proposition~2.6]{MV1} (also valid for $\kk$-vector lattices, with the same proof), $\Idc{p}$ is a surjective closed $0$-lattice homomorphism from~$\Idc\FL(I,\kk)$ onto~$\Idc{E}$.
The map $\psi_0\eqdef\gf\circ\Idc{p}$ is a closed $0$-lattice homomorphism from~$\Idc\FL(I_0,\kk)$ to~$L$.
By Theorem~\ref{T:ClosStep2}, together with the canonical isomorphism $\Idc\FL(I,\kk)\cong\Ops\kkp{I}$ given by Lemma~\ref{L:BakBey}, $\psi_0$ can be extended to a  surjective closed $0$-lattice homomorphism from~$\Idc\FL(J,\kk)$ onto~$L$. 
(Here we identify $\Idc\FL(I,\kk)$ with its canonical image into $\Idc\FL(J,\kk)$.)

Then $Q\eqdef\setm{x\in\FL(J,\kk)}{\psi(\seq{x})=0}$ is an \lidl\ of~$\FL(J,\kk)$, and further, denoting by~$q$ the canonical projection from $\FL(J,\kk)$ onto $F\eqdef\FL(J,\kk)/Q$, there exists a unique isomorphism $\gi\colon\Idc{F}\to L$ such that $\psi=\gi\circ\Idc{q}$ (cf. Wehrung \cite[Lemma~2.5]{MV1}).
The argument can be followed on Figure~\ref{Fig:Main}.
\begin{figure}[htb]
\begin{tikzcd}
\centering
\Idc\FL(I,\kk)
\arrow[r,hookrightarrow]
\arrow[d,"\Idc{p}",twoheadrightarrow]
\arrow[dr,"\psi_0"] &
\Idc\FL(J,\kk)\arrow[d,"\psi",twoheadrightarrow]
\arrow[dr,"\Idc{q}",twoheadrightarrow] & \\
\Idc{E}\arrow[r,"\gf","\text{closed}"'] & L & \Idc{F}\arrow[l,"\gi"',"\cong"]
\end{tikzcd}
\caption{Illustrating the proof of Theorem~\ref{T:Main}}
\label{Fig:Main}
\end{figure}

Now observe the following implications, for any $x\in\FL(I,\kk)$:
 \begin{align*}
 p(x)=0&\Leftrightarrow(\Idc{p})(\seq{x})=0\\
 &\Rightarrow\psi_0(\seq{x})=0&&
 (\text{because }\psi_0=\gf\circ\Idc{p})\\
 &\Leftrightarrow\psi(\seq{x})=0\\
 &\Leftrightarrow q(x)=0\,. 
 \end{align*}
Hence, by the First Isomorphism Theorem (for $\kk$-vector lattices), there exists a unique $\kk$,\lhom\ $f\colon E\to F$ such that $f\circ p=q\res_{\FL(I,\kk)}$.
It follows that
 \[
 \gi\circ\Idc{f}\circ\Idc{p}=
 \gi\circ\Idc(q\res_{\FL(I,\kk)})=\psi\res_{\Idc\FL(I,\kk)}
 =\gf\circ\Idc{p}\,.
 \]
Since the map~$\Idc{p}$ is surjective, it follows that $\gf=\gi\circ\Idc{f}$.
\end{proof}

It is then an easy matter to extend Theorem~\ref{T:Main} to its following diagram version.

\begin{theorem}\label{T:Main+}
Let~$\kk$ be a countable totally ordered division ring and let~$T$ be a forest in which every element has countable height.
Then every $T$-indexed commutative diagram $\vec{L}\eqdef\vecm{L_s,\gf_{s,t}}{s\leq t\text{ in }T}$ of countable \cn\ \dzlat{s}, with \emph{closed} $0$-lattice homomorphisms, is isomorphic to $\Idc\vec{E}$ for some $T$-indexed commutative diagram~$\vec{E}$ of $\kk$-vector lattices with $\kk$,\lhom{s}.
\end{theorem}

We illustrate Theorem~\ref{T:Main+} on Figure~\ref{Fig:Main1}, with~$o$ a minimal element of~$T$ and~$p$ an atom of~$T$.

\begin{figure}[htb]
\begin{tikzcd}
\centering
\Idc{E_{o}}\arrow[d,"\chi_{o}","\cong"']\arrow[r,"\Idc{f_{o,p}}"] &
\Idc{E_p}\arrow[d,"\chi_p","\cong"']\arrow[r] & \cdots & 
\cdots\arrow[r] & \Idc{E_t}\arrow[d,"\chi_t","\cong"']
\arrow[r] & \cdots\\
L_{o}\arrow[r,"\gf_{o,p}","\text{closed}"']
& L_p\arrow[r,"\text{closed}"'] & \cdots &
\cdots\arrow[r,"\text{closed}"'] & L_t\arrow[r,"\text{closed}"']
& \cdots
\end{tikzcd}
\caption{Illustrating Theorem~\ref{T:Main+}}
\label{Fig:Main1}
\end{figure}

\begin{proof}[Outline of proof]
An inductive argument within~$T$, similar, \emph{mutatis mutandis}, to the one of Plo\v{s}\v{c}ica and Wehrung \cite[Theorem~7.3]{MV2}, with Theorem~\ref{T:Main} used instead of \cite[Theorem~7.1]{MV2}, $E_s$ instead of $\FL(I_s,\kk)$, $f_{s,t}$ instead of~$\eta_{I_s,I_t}$, and all maps~$\chi_t$ now being isomorphisms.
At successor stages we apply Theorem~\ref{T:Main}, while at limit stages we apply the preservation of directed colimits by the functor~$\Idc$.
\end{proof}

The above argument yields in fact that \emph{every partial lifting of~$\vec{L}$, defined on a lower subset of~$T$, can be extended to a full lifting of~$\vec{L}$}.

\begin{corollary}\label{C:Main1}
The following are equivalent, for any countable totally ordered division ring~$\kk$ and any \dzlat~$L$ of cardinality at most~$\aleph_1$:
\begin{enumeratei}
\item\label{Reprkk}
There exists a $\kk$-vector lattice~$E$ such that $L\cong\Idc{E}$.

\item\label{Reprll}
There exists an Abelian \lgrp~$G$ such that $L\cong\Idc{G}$.

\item\label{cnCBD}
The lattice~$L$ is \cn\ and it has countably based differences.

\end{enumeratei}
\end{corollary}

\begin{proof}
\eqref{Reprll}$\Rightarrow$\eqref{cnCBD} is well known (and easy), see Cignoli \emph{et al.}~\cite[Thm.~2.2]{CGL}; see also Iberkleid \emph{et al.} \cite[Prop.~4.1.2]{IMM2011}, Wehrung~\cite[Lemma~10.1]{MV1}.
A similar argument also yields the implication \eqref{Reprkk}$\Rightarrow$\eqref{cnCBD} (the countability assumption on~$\kk$ can there be weakened by just saying that~$\kk$ has a countable cofinal sequence).

\eqref{cnCBD}$\Rightarrow$\eqref{Reprkk}.
Let~$L$ be a \cn\ \dzlat\ with countably based differences, of cardinality at most~$\aleph_1$.
Let us write $L=\setm{e_{\xi}}{\xi<\go_1}$.
The assumption that~$L$ has countably based differences means that there exists a sequence $\vecm{\sd_n}{n<\go}$ of binary operations on~$L$ such that for all $a,b,c\in L$, $a\leq b\vee c$ if{f} there exists $n<\go$ such that $a\sd_nb\leq c$.
For each $\ga<\go_1$, the closure~$L_{\ga}$ of $\set{0}\cup\setm{e_{\xi}}{\xi<\ga}$ under the lattice operations together with all~$\sd_n$ is a countable \cn\ $0$-sublattice of~$L$, and~$L$ is the ascending union of the~$L_{\ga}$.

Moreover, due to all~$L_{\xi}$ be closed under all operations~$\sd_n$, the inclusion map from~$L_{\ga}$ into~$L_{\gb}$ is a closed $0$-lattice embedding whenever $\ga<\gb<\go_1$.
Hence, by applying Theorem~\ref{T:Main+} to the well-ordered chain $\vec{L}\eqdef\vecm{L_{\ga}}{\ga<\go_1}$ of sublattices of~$L$, with the inclusion maps as transition maps (and $T\eqdef\go_1$), we obtain an $\go_1$-indexed commutative diagram $\vec{E}=\vecm{E_{\ga},f_{\ga,\gb}}{\ga\leq\gb<\go_1}$ of $\kk$-vector lattices such that $\Idc\vec{E}\cong\vec{L}$.
Setting $E\eqdef\varinjlim\vec{E}$, the preservation of all directed colimits by the functor~$\Idc$, together with the universal property of the colimit, yields $\Idc{E}\cong L$.

\eqref{cnCBD}$\Rightarrow$\eqref{Reprll} is just \eqref{cnCBD}$\Rightarrow$\eqref{Reprkk} with $\kk\eqdef\QQ$.
\end{proof}

Recall from Wehrung \cite[\S~9]{RAlg} that the direction \eqref{cnCBD}$\Rightarrow$\eqref{Reprkk} from Corollary~\ref{C:Main1} fails at every uncountable~$\kk$, even for countable lattices~$L$.

\section{Open problems}\label{S:Pbs}


Mart{\'\i}nez asks in \cite[Question~II]{Mart1973} for a characterization of all lattices~$\Idc{G}$ for~$G$ an Archimedean \lgrp.
Since all members of the class~$\cA(\gq,\vec{A})$ introduced in Wehrung \cite[Notation~12.2]{NonElt} are Archimedean \lgrp{s} whenever $\gq>\go$ (see the comments following Notation~12.2 in~\cite{NonElt}), \cite[Theorem~12.3]{NonElt} already yields a negative answer to Mart{\'\i}nez' question: \emph{the class of all isomorphic copies of~$\Idc{G}$, for~$G$ Archimedean, is not the class of models of any class of~$\scL_{\infty\gl}$ sentences for any~$\gl$; nor is it co-projective} (as defined on page~\pageref{co-proj}).

However, \emph{the lattices in~$\cA(\gq,\vec{A})$ have no top element}, because the arrows in the cube diagram~$\vec{A}$ from Wehrung~\cite{Ceva} are not unit-preserving.
This suggests the following problem, already hinted at in the comments following \cite[Corollary~12.9]{NonElt}.

\begin{problem}\label{Pb:ArchUnit}
Is the class of all isomorphic copies of~$\Idc{G}$, for~$G$ an Archimedean \lgrp\ \emph{with order-unit}, the class of all models of some~$\scL_{\infty\infty}$ sentence?
\end{problem}

Complete normality is a first-order sentence, while ``countably based differences'' is an~$\scL_{\go_1\go_1}$ sentence.
The following problem asks for an extension of Corollary~\ref{C:Main1} to real spectra of commutative unital rings.

\begin{problem}\label{Pb:RSAl1}
Denote by~$\Phi(A)$ the Stone dual of the real spectrum~$\Specr{A}$ of any commutative unital ring~$A$.
Can the isomorphic copies of the~$\Phi(A)$, with cardinality at most~$\aleph_1$, be characterized (within structures of size $\leq\aleph_1$) by an~$\scL_{\go_1\go_1}$ sentence?
\end{problem}

In more detail: it is well known that every~$\Phi(A)$ is a \cn\ bounded \dlat.
The main result of Wehrung~\cite{RAlg} states that every \emph{countable} \cn\ bounded \dlat\ can be obtained in this way.
As observed in Wehrung~\cite{MVRS}, this fails in a strong sense for lattices of cardinality~$\aleph_1$: \emph{there exists an Abelian \lgrp~$G$ of size~$\aleph_1$ such that~$\Idc{G}$ is not a homomorphic image of any~$\Phi(A)$}.
It is also observed in~\cite{MVRS} that a homomorphic image of some~$\Phi(A)$ of cardinality~$\aleph_1$ need not be isomorphic to any~$\Phi(A')$.

The cardinality assumption in Problem~\ref{Pb:RSAl1} cannot be omitted: by Mellor and Tressl~\cite{MelTre2012}, the class~$\cR$ of isomorphic copies of all~$\Phi(A)$ is not the class of models of any class of~$\scL_{\infty\gl}$ sentences for fixed~$\gl$.
This result is extended in Wehrung~\cite{AccProj}, by proving that~$\cR$ is not co-projective.

For any chain~$T$, denote by~$\cO(T)$ the sublattice of the powerset of~$T$ generated by all open intervals $\setm{x\in T}{x<a}$ and $\setm{x\in T}{x>a}$ for $a\in T$.
It is not hard to verify that~$\cO(T)$ is a \cn\ bounded \dlat.

\begin{problem}\label{Pb:O(T)}
Let~$T$ be a chain such that~$\cO(T)$ has countably based differences.
Does there exist an Abelian \lgrp~$G$ such that $\cO(T)\cong\Idc{G}$?
\end{problem}

By Corollary~\ref{C:Main1}, the answer to Problem~\ref{Pb:O(T)} is positive for $\card{T}\leq\aleph_1$.


\providecommand{\noopsort}[1]{}\def\cprime{$'$}
  \def\polhk#1{\setbox0=\hbox{#1}{\ooalign{\hidewidth
  \lower1.5ex\hbox{`}\hidewidth\crcr\unhbox0}}} \providecommand{\bysame}{\leavevmode\hbox to3em{\hrulefill}\thinspace}
\providecommand{\MR}{\relax\ifhmode\unskip\space\fi MR }
\providecommand{\MRhref}[2]{%
  \href{http://www.ams.org/mathscinet-getitem?mr=#1}{#2}
}
\providecommand{\href}[2]{#2}

\end{document}